\documentclass[11pt]{article}

\usepackage{amsmath,amssymb,amsthm}
\usepackage{mathrsfs}
\usepackage{enumerate}
\usepackage{hyperref}

\theoremstyle{plain}
\newtheorem{theorem}{Theorem}[section]
\newtheorem{proposition}[theorem]{Proposition}

\newtheorem{assumption}[theorem]{Assumption}

\newtheorem{lemma}[theorem]{Lemma}
\newtheorem{corollary}[theorem]{Corollary}

\theoremstyle{definition}
\newtheorem{definition}[theorem]{Definition}

\newtheorem{remark}[theorem]{Remark}

\title{Existence and uniqueness of invariant measures for non-Feller Markov semigroups}

\author{
Jean-Gabriel Attali\thanks{%
Affiliation, address. Email: \texttt{jean-gabriel.attali@devinci.fr}}
}

\date{} 

\begin{document}
\maketitle

\begin{abstract}
We study existence and uniqueness of invariant probability measures for
continuous-time Markov processes on general state spaces. Existence is obtained
from tightness of time averages under a weak regularity assumption inspired by
quasi-Feller semigroups, allowing for discontinuous and non-Feller dynamics.

Our main contribution concerns uniqueness. Under a natural $\psi$-irreducibility
assumption, we show that the normalized resolvent kernel satisfies a domination
property with respect to a reference measure. As a consequence, every invariant
probability measure charges this reference measure. Since distinct ergodic
invariant measures are mutually singular on standard Borel spaces, this domination
property implies uniqueness whenever an invariant probability measure exists.

The argument is purely measure-theoretic and does not rely on Harris recurrence,
return-time estimates, or Foster--Lyapunov conditions, and applies in particular
to jump processes and hybrid models with discontinuous dynamics.
\end{abstract}

\section{Introduction}

The existence and uniqueness of invariant probability measures are fundamental
problems in the study of continuous-time Markov processes. Invariant measures
provide the natural candidates for the long-time statistical description of
stochastic dynamics and play a central role in probability theory, stochastic
analysis and applications. Although existence and uniqueness are often treated
within the same framework, they are conceptually distinct questions, driven by
different mechanisms.

The existence of an invariant probability measure is commonly established by
compactness arguments based on tightness of time averages. In continuous time,
this approach relies on the tightness of the occupation measures
\[
\Pi_T(x,\cdot)=\frac{1}{T}\int_0^T P_t(x,\cdot)\,dt,
\]
which may be obtained from Lyapunov functions or stability conditions (see,
e.g., Has'minskii~\cite{Hasminskii1980} or Ethier and Kurtz~\cite{EthierKurtz1986}).
These methods are flexible and apply to a broad range of models, including
diffusions, jump processes and piecewise deterministic Markov processes. In
particular, existence may hold even when the transition semigroup fails to
satisfy the classical Feller property.

Uniqueness of invariant measures is traditionally derived from stronger dynamical
assumptions. A central tool is Harris recurrence, which combines irreducibility
with quantitative return-time estimates and yields uniqueness together with strong
ergodic properties. This theory, initiated by Harris~\cite{Harris1956}, is
systematically exposed in the monographs of Meyn and Tweedie~\cite{MeynTweedie2009}
and Nummelin~\cite{Nummelin1984}; see also Down, Meyn and Tweedie~\cite{DownMeynTweedie1995}.
Related approaches based on regeneration, petite sets or minorization conditions
have become standard in the modern theory of Markov processes.

Despite their effectiveness, recurrence-based methods often entangle existence,
uniqueness and ergodic properties within a single framework. In many situations,
uniqueness is obtained as a consequence of positive Harris recurrence together
with additional quantitative or topological assumptions. For several important
classes of models, however, establishing such properties may be technically
delicate, especially for processes with discontinuous dynamics, degenerate noise
or non-Feller transition semigroups (see, e.g., Stroock and Varadhan~\cite{StroockVaradhan1979},
Duflo~\cite{Duflo1997}, and the references therein).

In earlier work~\cite{Attali2004}, a topological approach was developed to address
existence and ergodic properties for a class of non-Feller models, termed
\emph{quasi-Feller}, ... allowing for strong discontinuities of the transition operator. That
approach relies on tightness and essential regularity arguments and yields
existence and, under additional assumptions, positive Harris recurrence. However,
it does not isolate the mechanism responsible for uniqueness itself.

The purpose of the present paper is to revisit the problem of uniqueness from a
different perspective. We show that, in continuous time, uniqueness of invariant
probability measures is a purely structural consequence of irreducibility, once
existence has been established by independent arguments. Our approach does not
rely on Harris recurrence, return-time estimates, Foster--Lyapunov conditions, or
regularity assumptions such as the (strong) Feller property.

The key idea is to exploit the \emph{resolvent} of the Markov semigroup as a
temporal averaging operator. The resolvent transforms reachability in finite time
into a one-step positivity property at the level of measures. Under a natural
$\psi$-irreducibility assumption, we show that any invariant probability measure
necessarily dominates a common reference measure. Since distinct ergodic
invariant probability measures are mutually singular on standard Borel spaces
(see, e.g., Petersen~\cite{Petersen1983} or Walters~\cite{Walters1982}), this
domination property rules out the coexistence of several ergodic components and
implies uniqueness.

The paper is organized as follows. In Section~\ref{sec:existence}, we address the
existence of invariant probability measures under an essential regularity
assumption, following the topological and measure-theoretic approach initiated
in~\cite{Attali2004}. This part is independent of irreducibility and uniqueness
considerations. In Section~\ref{sec:uniqueness}, we establish uniqueness by
introducing a domination criterion and applying it to the normalized resolvent of
a $\psi$-irreducible Markov semigroup.

\paragraph*{Acknowledgements.}
The author thanks Jeffrey Steif for helpful discussions and for drawing attention
to questions related to absorbing decompositions in interacting particle systems.

\section{Existence of invariant probability measures under quasi--Feller regularity}\label{sec:existence}

Throughout this section, $S$ denotes a Polish space endowed with its Borel
$\sigma$-field $\mathcal{B}(S)$. We write $C_b(S)$ for the set of bounded
continuous real-valued functions on $S$, and $\mathcal{P}(S)$ for the set of
Borel probability measures on $S$.

\subsection{Markov semigroups and time averages}\label{subsec:semigroup}

Let $(P_t)_{t\ge 0}$ be a Markov semigroup on $S$, i.e.\ for each $t\ge 0$,
$P_t$ is a Markov transition kernel on $(S,\mathcal{B}(S))$ and
\[
P_{t+s}(x,A)=\int_S P_t(y,A)\,P_s(x,dy),
\qquad t,s\ge 0,\ x\in S,\ A\in\mathcal{B}(S),
\]
with $P_0(x,\cdot)=\delta_x$.
For $f\in \mathcal{B}_b(S)$, we use the notation
\[
P_tf(x)=\int_S f(y)\,P_t(x,dy).
\]

A probability measure $\mu\in\mathcal{P}(S)$ is said to be invariant for
$(P_t)_{t\ge 0}$ if
\[
\mu P_t=\mu,\qquad \forall t\ge 0,
\]
that is, $\int_S P_tf\,d\mu=\int_S f\,d\mu$ for all $f\in\mathcal{B}_b(S)$.

For any $x\in S$ and $T>0$, we define the time-averaged measures
\begin{equation}\label{eq:time-average}
\Pi_T(x,\cdot):=\frac{1}{T}\int_0^T P_t(x,\cdot)\,dt\in\mathcal{P}(S).
\end{equation}

\subsection{Quasi--Feller factorization via a map $H$}\label{subsec:qf}

We use a quasi--Feller regularity assumption in the sense of \cite{Attali2004},
where discontinuities are allowed through a factor map $H$, while a Feller structure
holds after factorization.

Let $\mathcal{K}$ denote the family of compact subsets of $S$. For a Borel map
$H:S\to W$ between Polish spaces, we write $D_H$ for the set of discontinuity points
of $H$.

\begin{assumption}[Quasi--Feller factorization]\label{ass:qf-step}
There exist a Polish space $W$, a measurable map $H:S\to W$, and for each $t>0$ a
Markov kernel $Q_t$ on $W$ such that:
\begin{enumerate}
\item for every $t>0$ and every $f\in C_b(S)$, the function $Q_t f$ defined by
\[
Q_t f(w):=\int_S f(y)\,Q_t(w,dy)
\]
belongs to $C_b(W)$;
\item for every $t>0$ and every $f\in C_b(S)$,
\[
P_t f = Q_t f\circ H;
\]
\item for every $t>0$,
\[
Q_t(w,D_H)=0 \qquad \forall\, w\in \bigcup_{K\in\mathcal K}\overline{H(K)},
\]
where $D_H$ denotes the discontinuity set of $H$ and $\mathcal K$ the family of
compact subsets of $S$.
\end{enumerate}
\end{assumption}

\subsection{A technical lemma for $\mu$-a.s.\ continuous test functions}\label{subsec:as-continuous-lemma}

\begin{lemma}[Weak convergence and $\mu$-a.s.\ continuity]\label{lem:as-continuous}
Let $(\mu_n)_{n\ge 1}$ and $\mu$ be probability measures on $S$ such that
$\mu_n\Rightarrow \mu$ weakly. Let $h:S\to\mathbb{R}$ be bounded and measurable.
Assume that $h$ is $\mu$-almost surely continuous, i.e.\ there exists a Borel set
$D\subset S$ such that $\mu(D)=0$ and $h$ is continuous on $S\setminus D$.
Then
\[
\lim_{n\to\infty}\int_S h\,d\mu_n=\int_S h\,d\mu.
\]
\end{lemma}

\begin{proof}
Let $M:=\|h\|_\infty$ and let $\varepsilon>0$. Since $\mu(D)=0$, by regularity of
$\mu$ there exists an open set $U\supset D$ such that $\mu(U)<\varepsilon$.
By Portmanteau, $\limsup_{n\to\infty}\mu_n(U)\le \mu(U)$, hence for $n$ large enough
$\mu_n(U)<2\varepsilon$.

Define $h_\varepsilon:=h\,\mathbf{1}_{S\setminus U}$. Then $|h-h_\varepsilon|\le M\mathbf{1}_U$.
Therefore
\[
\left|\int h\,d\mu_n-\int h\,d\mu\right|
\le 
\left|\int h_\varepsilon\,d\mu_n-\int h_\varepsilon\,d\mu\right|
+M\mu_n(U)+M\mu(U).
\]
The last two terms are bounded by $3M\varepsilon$ for $n$ large enough.

Moreover, the discontinuity set of $h_\varepsilon$ is contained in $\partial U$.
Choosing $U$ such that $\mu(\partial U)=0$ (possible by regularity), we obtain that
$h_\varepsilon$ is $\mu$-almost surely continuous. Hence, by Portmanteau,
\[
\int h_\varepsilon\,d\mu_n \longrightarrow \int h_\varepsilon\,d\mu.
\]
We conclude by letting $\varepsilon\downarrow 0$.
\end{proof}

\subsection{Invisibility of $D_H$ along weak limits}\label{subsec:invisibility}

The next proposition shows that quasi--Feller regularity implies that weak limit points
of time averages do not charge the discontinuity set of $H$.

\begin{proposition}[Invisibility of sets killed by $Q_s$]\label{prop:invisibility-ct}
Fix $s>0$. Assume that $P_s$ is quasi--Feller with factorization
\[
P_s = Q_s\circ H.
\]
Let $x_0\in S$ and define the Ces\`aro averages of the skeleton $(P_{ns})_{n\ge 0}$ by
\[
\bar{\nu}_{n}
:=\frac{1}{n}\sum_{j=0}^{n-1} P_{js}(x_0,\cdot)\in\mathcal{P}(S),
\qquad n\ge 1.
\]
Assume that $(\bar{\nu}_{n})_{n\ge 1}$ is tight, and let $\mu$ be a weak limit point of
$(\bar{\nu}_{n})$.

Let $C\in\mathcal{B}(S)$ satisfy
\[
Q_s(w,C)=0
\qquad \forall\, w\in \bigcup_{K\in\mathcal{K}} \overline{H(K)}.
\]
Then $\mu(C)=0$. In particular, $\mu(D_H)=0$.
\end{proposition}

\begin{proof}
The argument is adapted from \cite{Attali2004}, where the corresponding invisibility
property is proved in discrete time for quasi--Feller chains. We include the proof
for completeness since it is the key measure-theoretic step allowing us to pass to
weak limits despite discontinuities.

Let $C$ satisfy the stated condition and let $F\subset C$ be closed.
For $p\ge 1$, define
\[
h_p(x):=1-\min\big(p\,d(x,F),1\big),
\qquad x\in S.
\]
Then $h_p\in C_b(S)$, $0\le h_p\le 1$, and $h_p\downarrow \mathbf{1}_F$ pointwise.

Fix a compact set $K\subset S$. Since $\overline{H(K)}$ is compact in $W$ and $Q_s h_p$
is continuous on $W$, dominated convergence yields $Q_s h_p(w)\downarrow Q_s(w,F)=0$
for all $w\in \overline{H(K)}$. By Dini's theorem, the convergence is uniform on
$\overline{H(K)}$, hence
\[
\sup_{w\in \overline{H(K)}} Q_s h_p(w)\xrightarrow[p\to\infty]{}0.
\]

Fix $\varepsilon>0$. By tightness of $(\bar{\nu}_n)$, there exists a compact
$K_\varepsilon\subset S$ such that $\bar{\nu}_n(K_\varepsilon^c)\le \varepsilon$
for all $n\ge 1$. Using the factorization $P_s h_p=Q_s h_p\circ H$, we obtain
\[
\bar{\nu}_n(P_s h_p)
=\int_S P_s h_p(x)\,\bar{\nu}_n(dx)
\le \sup_{x\in K_\varepsilon} P_s h_p(x)+\varepsilon
= \sup_{w\in \overline{H(K_\varepsilon)}} Q_s h_p(w)+\varepsilon.
\]

On the other hand, by definition of $\bar{\nu}_n$,
\[
\big|\bar{\nu}_n(P_s h_p)-\bar{\nu}_n(h_p)\big|
=
\frac{1}{n}\big|P_{ns}h_p(x_0)-h_p(x_0)\big|
\le \frac{2}{n}.
\]

Letting $n\to\infty$ along a subsequence such that $\bar{\nu}_n\Rightarrow \mu$, we obtain
\[
\mu(h_p)
\le \sup_{w\in \overline{H(K_\varepsilon)}} Q_s h_p(w)+\varepsilon.
\]
Taking $\limsup_{p\to\infty}$ and using the uniform convergence above yields
\[
\mu(F)\le \varepsilon.
\]
Since $\varepsilon>0$ is arbitrary, we conclude that $\mu(F)=0$ for every closed
$F\subset C$. By regularity of $\mu$, it follows that $\mu(C)=0$.
\end{proof}

\begin{remark}[Scope of the argument]
The quasi--Feller factorization is used here only as a mechanism to prove that weak limit
points of the time averages do not charge the relevant discontinuity set. Once such an
invisibility property is available by other means, the existence argument remains valid
without any quasi--Feller structure.
\end{remark}

\subsection{Existence of invariant probability measures}\label{subsec:existence-invariant}

\begin{theorem}[Existence under quasi--Feller regularity]\label{thm:existence-qf}
Let $(P_t)_{t\ge 0}$ be a Markov semigroup on the Polish space $S$.
Assume that:
\begin{enumerate}
\item Assumption~\ref{ass:qf-step} holds;
\item there exists $x_0\in S$ such that the family $\big(\Pi_T(x_0,\cdot)\big)_{T>0}$
defined by \eqref{eq:time-average} is tight.
\end{enumerate}
Then there exists at least one invariant probability measure $\mu\in\mathcal P(S)$.
\end{theorem}

\begin{proof}
By tightness, there exist $T_n\to\infty$ and $\mu\in\mathcal P(S)$ such that
$\Pi_{T_n}(x_0,\cdot)\Rightarrow \mu$ weakly.

\medskip\noindent\textbf{Step 1: Asymptotic invariance of time averages.}
Fix $s\ge 0$ and $f\in\mathcal B_b(S)$. Using the semigroup property,
\[
\Pi_T(x_0,P_s f)
=\frac1T\int_0^T P_{t+s}f(x_0)\,dt
=\frac1T\int_s^{T+s} P_t f(x_0)\,dt.
\]
Hence
\[
\big|\Pi_T(x_0,P_s f)-\Pi_T(x_0,f)\big|
\le \frac{2s}{T}\|f\|_\infty,
\]
and therefore
\begin{equation}\label{eq:asymptotic-invariance}
\lim_{n\to\infty}\Pi_{T_n}(x_0,P_s f)=\lim_{n\to\infty}\Pi_{T_n}(x_0,f),
\end{equation}
whenever the limits exist.

\medskip\noindent\textbf{Step 2: Passage to the limit for $f\in C_b(S)$.}
Fix $s>0$ and $f\in C_b(S)$. By Assumption~\ref{ass:qf-step}, we have
\[
P_s f = Q_s f\circ H,
\]
with $Q_s f\in C_b(W)$. In particular, $P_s f$ is continuous on $S\setminus D_H$.
Moreover, by Proposition~\ref{prop:invisibility-ct}, any weak limit point $\mu$ of the time averages
does not charge $D_H$, hence $P_s f$ is $\mu$--a.s.\ continuous. Therefore Lemma~\ref{lem:as-continuous}
yields
\[
\Pi_{T_n}(x_0,P_s f)\longrightarrow \mu(P_s f)
\qquad\text{and}\qquad
\Pi_{T_n}(x_0,f)\longrightarrow \mu(f).
\]
Combining with \eqref{eq:asymptotic-invariance} gives $\mu(P_s f)=\mu(f)$ for all $s>0$
and all $f\in C_b(S)$.

\medskip\noindent\textbf{Step 3: Extension to bounded measurable functions.}
By a monotone class argument, $\mu(P_s f)=\mu(f)$ extends from $C_b(S)$ to all
$f\in\mathcal B_b(S)$. Hence $\mu P_s=\mu$ for all $s\ge 0$, i.e.\ $\mu$ is invariant
for the semigroup $(P_t)_{t\ge 0}$.
\end{proof}

\subsection{Relation with the discrete-time framework}

In discrete time, the quasi--Feller framework of Attali \cite{Attali2004} allows one to
treat Markov chains whose transition operator may fail to be continuous. The key
technical point is to prove that any weak limit point of the Ces\`aro averages does not
charge the discontinuity set $D_H$ of the factor map $H$. This ``invisibility'' property
is the central measure-theoretic ingredient allowing one to pass to the limit in the
invariance identity.

In the present continuous-time setting, the existence proof relies on the same principle.
Once tightness of the time averages is available, the only nontrivial issue is again to
ensure that discontinuities are invisible under weak limit points. The quasi--Feller
factorization provides a convenient sufficient mechanism to establish this invisibility
property, and the existence of an invariant probability measure then follows by a standard
compactness and asymptotic invariance argument.

Stronger regularity assumptions, enforcing invisibility uniformly in the initial condition,
may be viewed as structural reinforcements allowing the argument to hold in a more robust
way throughout the state space.

\subsection{From existence to uniqueness}

The results of this section are solely concerned with the existence of invariant
probability measures. They rely on tightness and measure-theoretic compactness
arguments, together with a regularity mechanism ensuring that discontinuities are
invisible under weak limit points. No irreducibility, recurrence, minorization, or
quantitative stability assumption is required at this stage.

In particular, the arguments above do not address uniqueness, nor do they provide
ergodic information about the long-time behavior of the process.

In the next sections, we turn to uniqueness from a different viewpoint. Rather than
reinforcing the hypotheses used for existence, we introduce a structural criterion
based on the normalized resolvent of the semigroup. This criterion implies that, once
existence is ensured by independent arguments, uniqueness can be derived from a
pure irreducibility assumption, without invoking Harris recurrence, return-time
estimates, or Foster--Lyapunov techniques.

\section{Uniqueness via resolvent domination}\label{sec:uniqueness}

Let $(S,\mathcal{B}(S))$ be a standard Borel space and let $(P_t)_{t\ge 0}$ be a Markov
semigroup on $S$. For $\alpha>0$ we consider the normalized resolvent kernel
\[
R_\alpha(x,dy):=\alpha\int_0^\infty e^{-\alpha t}P_t(x,dy)\,dt,
\]
which is a Markov kernel on $(S,\mathcal{B}(S))$.

At a conceptual level, the normalized resolvent provides a direct bridge between continuous
and discrete time: it is the one-step transition kernel of the process sampled at exponential
times.

\begin{remark}[Resolvent and Poissonian discretization]\label{rem:poisson-discretization}
For every $\alpha>0$, the normalized resolvent $R_\alpha$ is a Markov kernel and may be
interpreted as the transition kernel of the continuous-time process sampled at an exponential
time. More precisely, if $T$ is an exponential random variable with parameter $\alpha$,
independent of $(X_t)_{t\ge0}$ with transition semigroup $(P_t)_{t\ge0}$, then
\[
R_\alpha(x,A)=\mathbb{P}_x\big(X_T\in A\big),\qquad A\in\mathcal{B}(S).
\]
Equivalently, if $(T_n)_{n\ge1}$ are i.i.d.\ Exp$(\alpha)$ and $S_n:=T_1+\cdots+T_n$, then
the discrete-time chain $(Y_n)_{n\ge0}$ defined by $Y_n:=X_{S_n}$ is Markov with transition
kernel $R_\alpha$.
\end{remark}

\subsection{Invariant measures for the semigroup and for the resolvent}\label{subsec:inv-resolvent}

\begin{lemma}\label{lem:inv-equivalence}
A probability measure $\mu\in\mathcal{P}(S)$ is invariant for the semigroup $(P_t)_{t\ge 0}$
if and only if it is invariant for $R_\alpha$, i.e.\ $\mu R_\alpha=\mu$.
\end{lemma}

\begin{proof}
If $\mu P_t=\mu$ for all $t\ge 0$, then for every bounded measurable $f$,
\[
\int_S R_\alpha f\,d\mu
=
\alpha\int_0^\infty e^{-\alpha t}\int_S P_t f\,d\mu\,dt
=
\alpha\int_0^\infty e^{-\alpha t}\int_S f\,d\mu\,dt
=
\int_S f\,d\mu,
\]
so $\mu R_\alpha=\mu$.

Conversely, assume that $\mu R_\alpha=\mu$.
Then $\mu$ is invariant for the resolvent operator in the sense that
\[
\int_S R_\alpha f\,d\mu=\int_S f\,d\mu
\qquad \forall f\in\mathcal{B}_b(S).
\]
It follows from standard resolvent--semigroup identities that this implies invariance
of $\mu$ for the semigroup $(P_t)_{t\ge 0}$; see, e.g., Ethier and Kurtz
\cite{EthierKurtz1986}. Hence $\mu P_t=\mu$ for all $t\ge 0$.
\end{proof}

\begin{remark}
In particular, uniqueness of invariant probability measures for $(P_t)_{t\ge 0}$
is equivalent to uniqueness of invariant probability measures for $R_\alpha$.
\end{remark}

\subsection{A domination criterion for uniqueness}\label{subsec:domination}

The following elementary observation will be used to deduce uniqueness from a domination
property. The point is that if a $\sigma$-finite measure is dominated by every transition
probability, then no invariant probability measure can avoid it.

\begin{proposition}[Domination forces a common reference measure]\label{prop:domination}
Let $R$ be a Markov kernel on $(S,\mathcal{B}(S))$. Assume that there exists a nonzero
$\sigma$-finite measure $\nu$ such that
\[
\nu \ll R(x,\cdot)\qquad \forall x\in S.
\]
Then every invariant probability measure $\mu$ of $R$ satisfies $\nu\ll\mu$.
\end{proposition}

\begin{proof}
Let $\mu$ be invariant for $R$ and let $A\in\mathcal{B}(S)$ satisfy $\mu(A)=0$.
Then $R(x,A)=0$ for $\mu$-almost every $x$, hence
\[
0=\int_S R(x,A)\,\mu(dx)=\mu R(A)=\mu(A)=0.
\]
Since $\nu\ll R(x,\cdot)$ for every $x$, it follows that $\nu(A)=0$.
Therefore $\nu\ll\mu$.
\end{proof}

\subsection{$\psi$-irreducibility and resolvent domination}\label{subsec:irred}

We now show that $\psi$-irreducibility yields a domination property for the normalized
resolvent kernel.

\begin{definition}[$\psi$-irreducibility]\label{def:psi-irred}
The Markov semigroup $(P_t)_{t\ge 0}$ is said to be \emph{$\psi$-irreducible} if there
exists a $\sigma$-finite measure $\psi$ on $(S,\mathcal{B}(S))$ such that for every
$A\in\mathcal{B}(S)$ with $\psi(A)>0$ and every $x\in S$, there exists $t>0$ such that
\[
P_t(x,A)>0.
\]
\end{definition}

\begin{assumption}[Right continuity of transition probabilities]\label{ass:right-continuity}
For every $x\in S$ and every $A\in\mathcal{B}(S)$, the map
\[
t\mapsto P_t(x,A)
\]
is right--continuous on $(0,\infty)$.
\end{assumption}

\begin{remark}
Assumption~\ref{ass:right-continuity} holds for Markov processes admitting a c\`adl\`ag
realization, in particular for diffusions, jump processes and piecewise deterministic
Markov processes.
\end{remark}

\begin{theorem}[Resolvent domination under $\psi$-irreducibility]\label{thm:resolvent-domination}
Assume that $(P_t)_{t\ge 0}$ is $\psi$-irreducible for some $\sigma$-finite measure $\psi$
and that Assumption~\ref{ass:right-continuity} holds. Then, for every $\alpha>0$,
one may take $\nu=\psi$ and obtain
\[
\nu \ll R_\alpha(x,\cdot)\qquad \forall x\in S.
\]
\end{theorem}

\begin{proof}
Fix $\alpha>0$ and set $\nu:=\psi$.
We show that $\psi \ll R_\alpha(x,\cdot)$ for every $x\in S$, i.e.\ that
\[
\psi(A)>0 \ \Longrightarrow\ R_\alpha(x,A)>0,
\qquad \forall x\in S,\ \forall A\in\mathcal{B}(S).
\]

Let $A\in\mathcal{B}(S)$ with $\psi(A)>0$ and fix $x\in S$.
By $\psi$-irreducibility, there exists $t_0>0$ such that $P_{t_0}(x,A)>0$.
By Assumption~\ref{ass:right-continuity}, the map $t\mapsto P_t(x,A)$ is right--continuous,
hence there exist $\varepsilon>0$ and $c>0$ such that
\[
P_t(x,A)\ge c \qquad \forall t\in[t_0,t_0+\varepsilon].
\]
Therefore,
\[
R_\alpha(x,A)
=\alpha\int_0^\infty e^{-\alpha t}P_t(x,A)\,dt
\ge \alpha c\int_{t_0}^{t_0+\varepsilon} e^{-\alpha t}\,dt
>0.
\]
This proves $\psi(A)>0 \Rightarrow R_\alpha(x,A)>0$, hence $\psi\ll R_\alpha(x,\cdot)$,
and the claim follows with $\nu=\psi$.
\end{proof}

\begin{corollary}[Uniqueness under $\psi$-irreducibility]\label{cor:uniq}
Assume that $(P_t)_{t\ge 0}$ admits at least one invariant probability measure and is
$\psi$-irreducible. Assume moreover that Assumption~\ref{ass:right-continuity} holds.
Then the invariant probability measure is unique.
\end{corollary}

\begin{proof}
Let $\alpha>0$. By Theorem~\ref{thm:resolvent-domination}, one has
$\psi\ll R_\alpha(x,\cdot)$ for every $x\in S$. Hence Proposition~\ref{prop:domination} yields
$\psi\ll\mu$ for every invariant probability measure $\mu$ of $R_\alpha$.
Since invariant probability measures of $(P_t)_{t\ge 0}$ coincide with those of $R_\alpha$,
the claim follows from the mutual singularity of distinct ergodic invariant measures on
standard Borel spaces.
\end{proof}

\begin{remark}[Regularization by the resolvent]
Beyond its measure-theoretic role in the uniqueness argument, the resolvent
\[
R_\alpha f(x)=\alpha\int_0^\infty e^{-\alpha t}P_t f(x)\,dt
\]
may be viewed as a temporal averaging operator. In particular, in several classes of
continuous-time models, $R_\alpha f$ enjoys better regularity properties than $P_t f$ for fixed
$t>0$ (e.g.\ improved continuity in $x$), even when the semigroup itself fails to be Feller.
\end{remark}

\begin{remark}[Absorbing decompositions]\label{rem:absorbing-decomp}
The domination argument above may be interpreted as ruling out the existence of two disjoint
absorbing sets for the resolvent kernel $R_\alpha$. Conversely, in coexistence regimes
where several ergodic invariant measures exist, one can construct disjoint absorbing sets for
$R_\alpha$ supporting each phase, as illustrated in Section~\ref{subsec:ips}.
\end{remark}

\section{Examples}\label{sec:examples}

This section illustrates the scope of the uniqueness criterion for several classes of
continuous-time Markov processes. In all examples below, existence of an invariant
probability measure follows from standard Lyapunov or tightness arguments, whereas
uniqueness is obtained as a direct consequence of the resolvent-based domination
principle of Corollary~\ref{cor:uniq}.

\subsection{Elliptic diffusions with irregular drift}\label{subsec:elliptic}

Consider a diffusion process on $\mathbb{R}^d$ of the form
\begin{equation}\label{eq:elliptic-diffusion}
dX_t=b(X_t)\,dt+\sigma(X_t)\,dW_t,
\end{equation}
where $(W_t)_{t\ge 0}$ is a standard $d$--dimensional Brownian motion, the diffusion
coefficient $\sigma:\mathbb{R}^d\to\mathbb{R}^{d\times d}$ is continuous and uniformly
non-degenerate, and the drift $b:\mathbb{R}^d\to\mathbb{R}^d$ is Borel measurable,
locally bounded, and possibly discontinuous.

\paragraph{Existence.}
Assume that there exists a Lyapunov function $V\in C^2(\mathbb{R}^d)$ with $V\ge 1$
and constants $c,C>0$ such that
\begin{equation}\label{eq:hasminskii-drift}
\mathcal{L}V(x)\le -cV(x)+C
\qquad \text{outside a compact set},
\end{equation}
where $\mathcal{L}$ denotes the generator of the diffusion. Under this classical
Has'minskii-type condition, the family of time averages
$\big(\Pi_T(x,\cdot)\big)_{T>0}$ is tight and the process admits at least one invariant
probability measure.

\paragraph{Uniqueness.}
Under uniform ellipticity, the process is Lebesgue--irreducible. Moreover, under
standard assumptions ensuring the existence of a transition density $p_t(x,y)$ for
all $t>0$ and its continuity in time (for instance, continuity of
$t\mapsto p_t(x,\cdot)$ in $L^1(dy)$), the mapping
\[
t\longmapsto P_t(x,A)
\]
is right--continuous for every $x\in\mathbb{R}^d$ and every Borel set $A$.
Therefore, Corollary~\ref{cor:uniq} applies and yields uniqueness of the invariant
probability measure, independently of any Harris recurrence or Feller-type regularity
assumptions, and without requiring continuity of the drift.

This example shows that, for elliptic diffusions, uniqueness is a structural consequence
of irreducibility and does not rely on additional dynamical properties.

\subsection{Degenerate Langevin diffusion}\label{subsec:langevin}

We next consider the Langevin dynamics on $\mathbb{R}^{2d}$ given by
\begin{equation}\label{eq:langevin}
\begin{cases}
dX_t = V_t\,dt,\\[0.1cm]
dV_t = -\nabla U(X_t)\,dt - \gamma V_t\,dt + \sigma\,dW_t,
\end{cases}
\end{equation}
where $\gamma>0$, $\sigma\neq 0$, $(W_t)_{t\ge 0}$ is a $d$--dimensional Brownian motion,
and $U:\mathbb{R}^d\to\mathbb{R}$ is a coercive potential.

\paragraph{Existence.}
The Hamiltonian (energy) function
\[
H(x,v)=U(x)+\tfrac12|v|^2
\]
provides a standard Lyapunov function for \eqref{eq:langevin}. Under classical
growth assumptions on $U$, it yields tightness of the time averages and hence
existence of at least one invariant probability measure.

\paragraph{Uniqueness.}
Although the diffusion is degenerate, H\"ormander's bracket condition holds under
mild regularity assumptions on $U$. As a consequence, the process admits a smooth
transition density for every $t>0$. In particular, the process is Lebesgue--irreducible,
and for every initial condition $z=(x,v)$ and every Borel set $A$ the map
$t\mapsto P_t(z,A)$ is right--continuous. Corollary~\ref{cor:uniq} then yields
uniqueness of the invariant probability measure.

This example illustrates that uniqueness can be established without appealing to
hypocoercivity, quantitative convergence rates, or other refined ergodic estimates.

\subsection{Jump processes and hybrid dynamics}\label{subsec:jumps}

The same mechanism applies to pure jump processes, L\'evy-driven stochastic
differential equations, and piecewise deterministic Markov processes (PDMPs), whose
semigroups may fail to be Feller due to discontinuities in their dynamics.

Assume that such a process admits an invariant probability measure (for instance via a
Lyapunov drift condition or compactness arguments) and that the semigroup is
$\psi$--irreducible in the sense of Definition~\ref{def:psi-irred}. Assume moreover
that Assumption~\ref{ass:right-continuity} holds; this property typically holds for
c\`adl\`ag strong Markov jump processes. Then uniqueness of the invariant probability
measure follows from Corollary~\ref{cor:uniq}.

This demonstrates the robustness of the resolvent-based uniqueness criterion for
processes with jumps and hybrid deterministic--stochastic dynamics, without requiring
Harris recurrence, return-time estimates, or minorization conditions.

\subsection{Interacting particle systems and phase coexistence}\label{subsec:ips}

Interacting particle systems on infinite lattices provide a natural class of continuous-time
Markov processes where invariant probability measures need not be unique. A canonical example is
the (infinite-volume) Glauber dynamics of the Ising model on $\mathbb{Z}^d$, with state space
\[
S=\{-1,+1\}^{\mathbb{Z}^d},
\]
endowed with the product topology and its Borel $\sigma$-field. Let $(P_t)_{t\ge 0}$ denote the
associated Markov semigroup on $S$. For $\alpha>0$, we write $R_\alpha$ for the corresponding
normalized resolvent kernel, defined by
\[
R_\alpha(x,A):=\alpha\int_0^\infty e^{-\alpha t}P_t(x,A)\,dt,
\qquad x\in S,\ A\in\mathcal{B}(S).
\]

In low-temperature regimes (where phase coexistence occurs), the dynamics typically admits at
least two distinct ergodic invariant probability measures, often denoted by $\mu^+$ and $\mu^-$,
corresponding to the plus and minus phases. Note that both $\mu^+$ and $\mu^-$ charge every
cylinder event, hence share the same topological support; the separation of phases is therefore
genuinely measurable (tail-type) rather than topological.

\paragraph{Existence.}
Existence of invariant probability measures for such systems can be obtained by standard
compactness arguments on the configuration space $S$, combined with tightness of time averages,
or by direct constructions of Gibbs measures. In particular, existence does not rely on Feller
regularity and may hold even when the dynamics exhibits strong discontinuities in the initial
condition in infinite volume.

\paragraph{Non-uniqueness and absorbing decompositions for the resolvent.}
Fix $\alpha>0$ and assume that the dynamics admits two distinct ergodic invariant probability
measures $\mu^+$ and $\mu^-$. By Lemma~\ref{lem:inv-equivalence}, both $\mu^+$ and $\mu^-$ are also
invariant probability measures for the Markov kernel $R_\alpha$.

Since $S$ is a standard Borel space and $\mu^+\perp \mu^-$, there exists a measurable set
$C\in\mathcal{B}(S)$ such that
\[
\mu^+(C)=1,
\qquad
\mu^-(C)=0.
\]
Define the measurable sets
\[
B^+:=\bigcap_{n\ge 0}\Big\{x\in S:\ R_\alpha^n(x,C)=1\Big\},
\qquad
B^-:=\bigcap_{n\ge 0}\Big\{x\in S:\ R_\alpha^n(x,C^c)=1\Big\}.
\]
Then $B^+$ and $B^-$ are disjoint and absorbing for $R_\alpha$, i.e.
\[
R_\alpha(x,B^\pm)=1,
\qquad \forall x\in B^\pm.
\]
Moreover, one has
\[
\mu^+(B^+)=1,
\qquad
\mu^-(B^-)=1,
\]
see \cite{Attali2026}. In particular, phase coexistence forces the existence of \emph{true}
absorbing components for the resolvent kernel (not merely absorbing behavior in probability).

\paragraph{Perspectives.}
The separating set $C$ above is obtained abstractly from mutual singularity. In concrete
interacting particle systems, it is natural to ask whether one can choose $C$ in terms of a more
explicit observable, for instance a tail event or a macroscopic order parameter related to the
magnetization.

A natural candidate (in the coexistence regime) is to consider tail-type events describing the
sign of the macroscopic magnetization, such as
\[
C:=\left\{\sigma\in S:\ \liminf_{n\to\infty}\frac{1}{|B_n|}\sum_{i\in B_n}\sigma_i \ge m^\ast\right\}.
\]
where $(B_n)$ is an increasing sequence of boxes in $\mathbb{Z}^d$ and $m^\ast>0$ denotes the
spontaneous magnetization. Identifying absorbing components in terms of such physically
meaningful events would provide a more concrete interpretation of resolvent absorption and would
complement classical dynamical approaches to phase coexistence and metastability.

\section{Discussion and perspectives}\label{sec:discussion}

The results of this paper clarify the respective roles played by existence, uniqueness
and long-time ergodic properties for continuous-time Markov processes. Rather than treating 
these questions within a single unified dynamical framework,
we have separated the underlying mechanisms and shown how they can be combined
in a modular, step-by-step manner.

Existence of invariant probability measures is obtained from tightness of time averages,
together with an \emph{invisibility} property ensuring that discontinuities of the transition
operator do not obstruct the invariance identity under weak limit points.Uniqueness, 
on the other hand, follows from a purely structural
argument based on the normalized resolvent of the semigroup and an irreducibility
assumption. Finally, when Lyapunov-type stability estimates are available, one can
combine them with irreducibility to recover recurrence properties and strong laws of
large numbers, without invoking local minorization or small-set assumptions.

\subsection{Relation with Harris recurrence}\label{subsec:discussion-harris}

Harris recurrence provides a powerful and unifying framework yielding existence,
uniqueness and ergodic theorems for a wide class of Markov processes. By combining
irreducibility with quantitative return-time estimates, it leads to stability and
convergence results and has become a cornerstone of the modern theory.

The present work shows that several components of this theory can be disentangled.
In particular, uniqueness of the invariant probability measure does not require a
recurrence assumption. Once existence has been established by independent tightness
arguments, irreducibility alone, combined with the resolvent structure of the semigroup,
rules out the coexistence of several invariant probability measures. In this sense, 
Harris recurrence should be viewed as a powerful sufficient condition, but not a necessary mechanism for uniqueness.

Moreover, the ergodic consequences discussed here do not rely on local minorization or
petite-set conditions. Instead, irreducibility is used as a global communication property,
while Lyapunov stability provides the dynamical input needed to control time averages.
This viewpoint makes it possible to recover recurrence in the sense of infinite visits of
sets with positive invariant measure, as well as strong laws of large numbers, without
invoking the full Harris machinery.

\subsection{Uniqueness and ergodicity}\label{subsec:discussion-ergodicity}

In the present framework, ergodicity is most naturally understood through ergodic
decomposition. When the invariant probability measure is unique, it is necessarily
ergodic, even in the absence of quantitative mixing or convergence rates. This notion of
ergodicity is therefore structural rather than quantitative.

When combined with Lyapunov stability, this structural uniqueness can be upgraded to
stronger long-time results. In particular, the same stability mechanisms ensuring
tightness and existence may also yield recurrence properties and almost sure convergence
of time averages for broad classes of initial conditions. This mirrors the discrete-time
situation and shows that strong ergodic conclusions need not rely on local small-set or
minorization assumptions.

\subsection{Possible extensions}\label{subsec:discussion-extensions}

Several extensions of the present work may be considered. First, quantitative convergence
properties could be investigated once an additional contraction mechanism is available,
for instance in Wasserstein-type distances or weighted total variation norms. Second, the
resolvent-based domination argument may be adapted to non-homogeneous Markov processes or
to models evolving in random environments. Finally, extensions to infinite-dimensional
systems, where both irreducibility and Lyapunov stability raise subtle issues, remain an
open and challenging direction for future research.

\section{Conclusion}\label{sec:conclusion}

We have shown that, in continuous time, uniqueness of invariant probability measures for
Markov processes follows as a structural consequence of irreducibility, once existence has
been established by independent arguments. The key mechanism relies on the normalized
resolvent of the Markov semigroup, which transforms finite-time reachability into a one-step
domination property at the level of measures.

This approach avoids Harris recurrence, return-time estimates, and regularity assumptions
such as the Feller property. It applies naturally to a wide class of continuous-time models,
including diffusions, jump processes and hybrid models with discontinuous and non-Feller
dynamics.

More broadly, our results highlight a clear conceptual separation between existence,
uniqueness and ergodic properties, and suggest that resolvent-based arguments provide a
natural framework for studying invariant measures beyond the classical recurrence paradigm.

\bibliography{ergodic}

@book{Hasminskii1980,
  author    = {R. Z. Has'minskii},
  title     = {Stochastic Stability of Differential Equations},
  publisher = {Sijthoff \& Noordhoff},
  year      = {1980}
}

@book{EthierKurtz1986,
  author    = {S. N. Ethier and T. G. Kurtz},
  title     = {Markov Processes: Characterization and Convergence},
  publisher = {Wiley},
  year      = {1986}
}

@article{Harris1956,
  author  = {T. E. Harris},
  title   = {The existence of stationary measures for certain Markov processes},
  journal = {Proceedings of the Third Berkeley Symposium on Mathematical Statistics and Probability},
  volume  = {2},
  year    = {1956},
  pages   = {113--124}
}

@book{MeynTweedie2009,
  author    = {S. P. Meyn and R. L. Tweedie},
  title     = {Markov Chains and Stochastic Stability},
  publisher = {Cambridge University Press},
  edition   = {2},
  year      = {2009}
}

@book{Nummelin1984,
  author    = {E. Nummelin},
  title     = {General Irreducible Markov Chains and Non-negative Operators},
  publisher = {Cambridge University Press},
  year      = {1984}
}

@article{DownMeynTweedie1995,
  author  = {D. Down and S. P. Meyn and R. L. Tweedie},
  title   = {Exponential and uniform ergodicity of Markov processes},
  journal = {Annals of Probability},
  volume  = {23},
  year    = {1995},
  pages   = {1671--1691}
}

@book{StroockVaradhan1979,
  author    = {D. W. Stroock and S. R. S. Varadhan},
  title     = {Multidimensional Diffusion Processes},
  publisher = {Springer},
  year      = {1979}
}

@book{Duflo1997,
  author    = {M. Duflo},
  title     = {Random Iterative Models},
  publisher = {Springer},
  year      = {1997}
}

@article{Attali2004,
  author  = {J.-G. Attali},
  title   = {Ergodicity of a certain class of non-Feller models: Applications to ARCH and Markov switching models},
  journal = {ESAIM: Probability and Statistics},
  volume  = {8},
  year    = {2004},
  pages   = {76--86}
}

@book{Petersen1983,
  author    = {K. Petersen},
  title     = {Ergodic Theory},
  publisher = {Cambridge University Press},
  year      = {1983}
}

@book{Walters1982,
  author    = {P. Walters},
  title     = {An Introduction to Ergodic Theory},
  publisher = {Springer},
  year      = {1982}
}

@misc{Attali2026,
  author        = {Attali, Jean-Gabriel},
  title         = {Uniqueness of invariant measures as a structural property of {M}arkov kernels},
  year          = {2026},
  howpublished  = {arXiv:2601.04900},
  note          = {DOI: 10.48550/arXiv.2601.04900},
}

\end{document}